\documentclass[a4paper,12pt]{article}
\usepackage[top=2.5cm,bottom=2.5cm,left=2.5cm,right=2.5cm]{geometry}
\usepackage{cite, amsmath, amssymb}

\usepackage[english]{babel}
\usepackage{amsthm}

\usepackage[small,bf,hang]{caption} 

\usepackage{graphicx}

\usepackage{hyperref}
               
\usepackage{gensymb}
\usepackage{multirow}
\def\rank{\mathrm{rank}\,}

\newcommand*{\R}{\mathbf{R}}
\newcommand*{\F}{\mathbf{F}}
\newcommand*{\adj}{\mathop{\mathrm{adj}}\nolimits}
\newtheorem{theorem}{Theorem}[section]
\theoremstyle{plain}

\newtheorem{lemma}[theorem]{Lemma}
\newtheorem{corollary}[theorem]{Corollary}

\graphicspath{{figures/}}

\usepackage{color}

\begin{document}


\title{\textbf{Generation and properties of nut graphs}}

\author{\bigskip\textbf{Kris Coolsaet$^a$, Patrick W.\ Fowler$^b$, Jan Goedgebeur$^a$}\\
$^a$\textit{Department of Applied Mathematics, Computer Science \& Statistics}\\
\textit{Ghent University}\\
\textit{Krijgslaan 281-S9, 9000 Ghent, Belgium}\\
\medskip
\texttt{kris.coolsaet@ugent.be}, \texttt{jan.goedgebeur@ugent.be}\\
$^b$\textit{Department of Chemistry}\\  
\textit{University of Sheffield}\\  
\textit{ Brook Hill, Sheffield S3 7HF, United Kingdom}\\
\texttt{p.w.fowler@sheffield.ac.uk}
}

\date{} 
\maketitle

\vspace*{-10mm}

\begin{center}
\,
\end{center}

\begin{abstract}
A \emph{nut graph} is a graph on at least 2 vertices whose adjacency matrix has nullity~1 and for which non-trivial kernel vectors do not contain a zero. 
Chemical graphs are connected, with maximum degree at most three.
We present a new algorithm for the exhaustive generation of non-isomorphic nut graphs. Using this algorithm, we determined all nut graphs up to 13 vertices and all chemical nut graphs up to 22 vertices. 
Furthermore, we determined all nut graphs among the cubic polyhedra up to 34 vertices and all nut fullerenes up to 250 vertices.

Nut graphs are of interest in chemistry of conjugated systems, in models of electronic structure, radical reactivity 
and molecular conduction. 
The relevant mathematical properties of chemical nut graphs are the position of the zero eigenvalue in the 
graph spectrum, and the dispersion in magnitudes of kernel eigenvector entries ($r$: the ratio of maximum to minimum magnitude
of entries). 
Statistics are gathered on these properties for all the nut graphs generated here.
We also show that all chemical nut graphs
have $r \ge 2$ and that there is at least one chemical nut graph
with $r = 2$ for every order $n \ge 9$ (with the exception of $n = 10$).
\end{abstract}

\baselineskip=0.30in 


\section{Introduction}

A \emph{nut graph} is a graph  of at least 2 vertices whose adjacency matrix has nullity 1 (i.e.\  rank $n-1$ where $n$ is the order of the graph) and for which non-trivial kernel
vectors do not contain a zero.

The topic of nut graphs, introduced by Sciriha and Gutman in~\cite{gutman1996graphs,sciriha1998nut}, is one that emerged from pure mathematics (linear algebra and graph theory), but which turns out to have natural connections with chemical theory in 
at least three distinct areas: electronic structure theory, the chemical reactivity of radicals and, 
perhaps more surprisingly, the theory of molecular conduction. The applications have generated new mathematical questions, and these in turn have implications for the scope of the chemical applications. This will be discussed in 
more detail in Section~\ref{subsect:nut_chem_props}.
In this connection, we note that a
\emph{chemical graph} is a connected graph with maximum degree at most 3. This definition is 
motivated by the use of graph theory in chemistry to describe electronic structure of unsaturated 
carbon networks (H\"uckel theory~\cite{Streitwieser1962}), where vertices represent carbon atoms with bonds to at most three carbon neighbours.

The smallest nut graphs have seven vertices; 
there are three seven-vertex nuts and they are shown in Figure~\ref{fig:smallest_nuts}. 
Nuts form a subset of {\it core} graphs~\cite{sciriha1998construction,sciriha2009maximal}: a core graph is singular and has every vertex appearing with non-zero entry in some eigenvector belonging to the nullspace.  A useful property of both nut and core graphs is that deletion of any vertex reduces the nullity by one.
It is also useful to note that
nut graphs are non-bipartite and have no leaves~\cite{sciriha1998nut}.
The smallest {\it chemical} nut graph has nine vertices and is shown in 
Figure~\ref{fig:smallest_chemnut}.

\begin{figure}[h!t]
	\centering
	\includegraphics[width=0.8\textwidth]{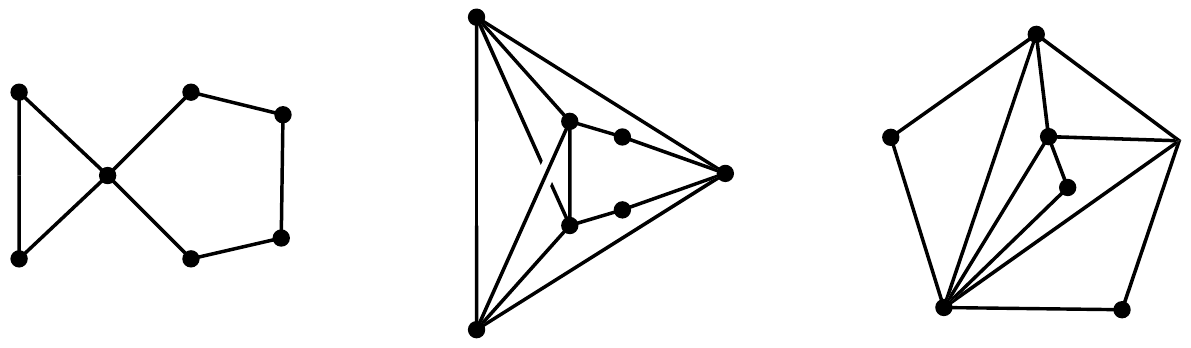}
	\caption{The smallest nut graphs.}
	\label{fig:smallest_nuts}
\end{figure}

\begin{figure}[h!t]
	\centering
	\includegraphics[width=0.35\textwidth]{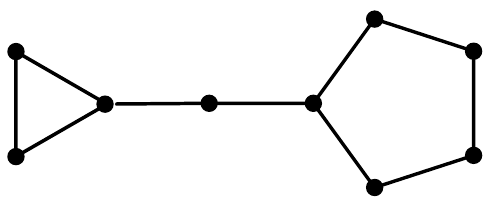}
	\caption{The smallest chemical nut graph.}
	\label{fig:smallest_chemnut}
\end{figure}

There are various published constructions for expanding a nut graph, for example by adding duplicate vertices, or expanding edges to paths of appropriate parity, from which it is clear that arbitrarily large chemical nut graphs exist~\cite{Sciriha2008}.

In~\cite{fowler2014omni} Fowler et al.\ determined all nut graphs up to 10 vertices,
and all chemical nut graphs up to 16 vertices, respectively. Furthermore in~\cite{sciriha2007nonbonding} Sciriha and Fowler also determined all nut graphs among the cubic polyhedra up to 24 vertices.
They also determined all nut fullerenes up to 120 vertices and showed that there are no nut IPR fullerenes up to at least 150 vertices~\cite{sciriha2007nonbonding}. (A \textit{fullerene}~\cite{kroto_85} is a cubic polyhedron where all faces have size 5 or 6.)

In this article we present a specialised generation algorithm for nut graphs and using this algorithm we are able to expand significantly these lists of nut graphs. 

The paper is organised as follows. In Section~\ref{section:generation_nuts} we present our generation algorithm for nut graphs. In Section~\ref{subsect:nut_counts} we present the complete lists of nut graphs which were generated by our implementation of this algorithm. Finally, in Section~\ref{subsect:nut_chem_props} we describe the results of our computations of chemically relevant properties on the lists of nut graphs.

\section{Generation of nut graphs}
\label{section:generation_nuts}

Several techniques can be used to determine whether a graph is a nut graph. A straightforward approach is to use the eigenvalues and eigenvectors of the corresponding adjacency matrix. The graph is a nut graph if and only if it has at least two vertices and 
there is exactly one eigenvalue equal to zero and the corresponding eigenvector has no zero entries. 

Although fast numerical algorithms for the determination of eigenvalues and eigenvectors of symmetric matrices exist, they are not ideal for our purposes because the floating point approximations used in computer implementations of these methods suffer from an inherent inaccuracy problem: it is never certain whether a result that `looks like' zero (perhaps because it coincides with zero up to 12 decimal places) corresponds to an actual zero, and conversely whether a result that seems different from zero might not be a real zero suffering from rounding errors.

For many problems of a numerical nature this is not an issue, because the value that is computed is a continuous function of the input values. 
Unfortunately, the rank of a matrix (and hence the property of being a nut graph) does not belong to this category.  
In our case it would help if we could determine in advance to what accuracy eigenvalues need to be computed, that is to say, 
if it were possible to calculate a lower bound on the minimal non-zero eigenvalue of a graph in advance. We 
are not aware of any theoretical results in this direction.

That inaccurate computations could indeed lead to false conclusions is illustrated by a literature example that is quite similar to the problem at hand: in a search for graphs with eigenvalue
$\sqrt{8}$
to test a conjecture made by Dias~\cite{dias1995structural},
examples of false-positive graphs were found~\cite{fowler2010counterexamples} at $n = 18$, i.e.\  graphs with a single eigenvalue not equal to $\sqrt{8}$  but coinciding with $\sqrt{8}$ to 10 places of decimals. The graph $G_{18}$ in Figure~\ref{fig:false_pos} has largest eigenvalue
$\lambda_1 = \sqrt8-\epsilon$, with $\epsilon \approx 5 \times 10\sp{-11}$.
This and other false positives in the search on the Dias conjecture~\cite{dias1995structural} were detected by the
informal method of checking a {\lq danger zone\rq} around the target eigenvalue, based on a cautious estimate of the numerical error of the eigenvalue routines (combined with the 
knowledge that a true eigenvalue $\sqrt{8}$ would be paired with another at $-\sqrt{8}$). 
\begin{figure}[h!t]
	\centering
	\includegraphics[width=0.6\textwidth]{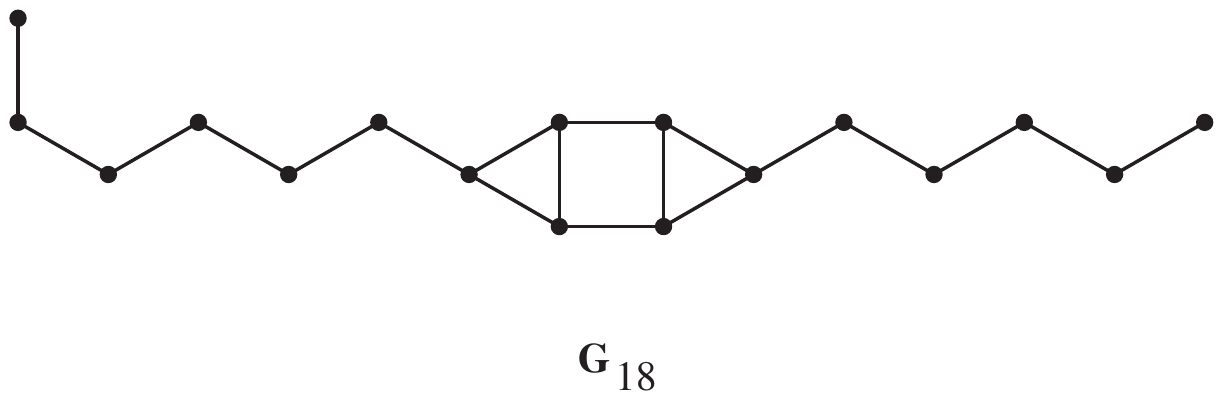}
	\caption{An $18$-vertex graph with Perron eigenvalue $\sqrt{8}-\epsilon$;
by taking the cartesian product of this graph with the star on 9 vertices,
a graph with a full eigenvector at eigenvalue $-\epsilon$ is produced.
In this case, $\epsilon \approx 5 \times 10\sp{-11}$.}
	\label{fig:false_pos}
\end{figure}
It is easy to use such an example to generate a graph with a full eigenvector that corresponds to a near-zero eigenvalue: simply take the direct product of $G_{18}$ with $S_9$, the star on nine vertices, which has smallest eigenvalue $-\sqrt{8}$. The product $G_a \square G_b$
has eigenvalues $\lambda_a + \lambda_b$, where $\lambda_a$ and $\lambda_b$ run over the spectra of $G_a$ and $G_b$, respectively, and entries in non-degenerate eigenvectors of $G_a \square G_b$ are simple products of the eigenvector entries of the starting graphs. In the present case, 
the $162$-vertex graph
$G_{18} \square S_9$ therefore has a non-degenerate eigenvalue 
$-\epsilon$ formed from the sum of the Perron eigenvalue of $G_{18}$ and the anti-Perron $-\sqrt8$ eigenvalue of the star. As both eigenvectors are full, this vector is also full. ($G_{18} \square S_9$ also has $14$ true zero eigenvalues, arising from the nullities of $2$ and $7$ of $G_{18}$ and $S_9$, respectively). 

A well known and more formal method for countering the inaccuracy problem is to use multiprecision integer arithmetic. Indeed, there are various methods of checking the number of zero eigenvalues of a candidate graph, and of determining the relative values of the entries in an
eigenvector that corresponds to a unique zero,
using only integer arithmetic, see for instance Longuet-Higgins~\cite{longuet1950some} 
elaborated by {\v Z}ivkovi{\' c}~\cite{zivkovic1972calculation}. These methods are easiest to implement for eigenvalues that are integers but the algorithms can be extended to algebraic integers in general, such as the eigenvalue $\sqrt 8$ mentioned above. (Note that eigenvalues of graphs are always algebraic integers.) 

The main disadvantage of these algorithms is that they are much slower than the classical methods, because their speed depends on the size of the numbers involved, and these grow quickly with larger graph orders. Tests indicate that for the problem at hand the multiprecision method is slower by at least an order of magnitude than the generation algorithm which we eventually used (cf.\ Section~\ref{subsection:modulo_p}).

\subsection{Properties of nut graphs}
For the reasons mentioned above, we compute the rank of a matrix directly without previous computation of the eigenvalues. Our algorithm is based on the following properties of adjacency matrices of nut graphs.

\begin{theorem}
\label{theo:nut}
Consider a graph $\Gamma$ with adjacency matrix
\[
A= 
\left(\!
\begin{array}{c|c}
B & b^T \\ \hline
b & 0 \\
\end{array}
\!\right).
\]
Then $\Gamma$ is a nut graph if and only if
\begin{enumerate}
\item $B$ is non-singular, 
\item $bB^{-1}b^T = 0$, and
\item $bB^{-1}$ has no zero entries.
\end{enumerate}
\end{theorem}
\begin{proof}
First assume that $B$ is non-singular. We multiply $A$ on the left with a non-singular matrix, as follows
\[
\left(\!
\begin{array}{c|c}
B^{-1} & 0 \\ \hline
-bB^{-1}  & 1 \\
\end{array}
\!\right)
\left(\!
\begin{array}{c|c}
B & b^T \\ \hline
b & 0 \\
\end{array}
\!\right)
=
\left(\!
\begin{array}{c|c}
1 & B^{-1}b^T \\ \hline
0 & -bB^{-1}b^T \\
\end{array}
\!\right)
\]
If $-bB^{-1}b^T \ne 0$ then the right hand matrix, and hence also $A$, has full rank. Otherwise, both matrices have rank $n-1$. 

In the latter case, consider the vector $(bB^{-1}\mid -1)$. We have
\[(bB^{-1}\mid -1) \left(\!
\begin{array}{c|c}
B & b^T \\ \hline
b & 0 \\
\end{array}
\!\right)
=
(0 \mid bB^{-1}b^T) = (0\mid 0).\]
Hence $(bB^{-1}\mid -1)$ is a (non-trivial) eigenvector of $A$. If this vector contains no zero entries, 
the resulting graph is a nut graph.

Conversely, assume $\Gamma$ is a nut graph and consider a kernel vector of $A$. Because the last entry of this kernel vector is non-zero, we may always multiply the vector with a scalar to obtain a kernel vector of the form
 $(x \mid -1)$. From $(x \mid -1)A = 0$ we find $xB = b$ and $xb^T = 0$. 

If $B$ has an inverse, we find  
$x = bB^{-1}$ and  the theorem follows.
Otherwise, let $y$ denote a non-trivial kernel vector of $B$. Then $(y\mid 0) A = (0\mid yb^T) = (0\mid yBx^T) = (0\mid 0)$. Hence $(y\mid 0)$ is a kernel vector of $A$ with an entry equal to zero, and $\Gamma$ cannot be a nut graph.
\end{proof}

Alternative proofs of the properties in  Theorem \ref{theo:nut}
can for instance be found in \cite{Sciriha2008}.
\begin{lemma}
\label{lemma-adj}
A graph with adjacency matrix $A$ is a nut graph if and only if
$\det A = 0$ and $\adj A$ has no zero entries.
\end{lemma}
\begin{proof}
A matrix $A$ has nullity 1 if and only if $\det A = 0$ and $\adj A\ne 0$.
Since $A \adj A = (\det A)1 = 0$, the columns of $\adj A$ are eigenvectors
of $A$ (possibly 0). Since $A$ has rank $n-1$ every column of $\adj A$
must be a multiple of a fixed non-trivial eigenvector $x^T$ of $A$. 

Write $\alpha_i x$ for the 
$i$th column of $\adj A$. Then $y=(\alpha_1,\ldots,\alpha_n)$ is non-zero. Moreover, every row of $\adj A$ is a multiple of $y$. In fact: $\adj A = x^Ty$.

As $\adj A$ is symmetric, the rows
of $\adj A$ are eigenvectors of $A$ and hence so is $y$. If $\Gamma$ is not a nut graph, then $\alpha_i=0$ for at least one index $i$ and $\adj A$ contains at least one zero column (and row). If $\Gamma$ is a nut graph,
then neither $x$ nor $y$ contains a zero entry, and hence all entries
of $\adj A$ are non-zero.
\end{proof}

\subsection{Generation algorithm}
\label{subsection:generation_algorithm}
The algorithm that was developed here
for generating all nut graphs of a given order~$n$ is an adaptation of the general canonical construction path method for isomorph-free generation of graphs, pioneered by McKay~\cite{mckay_98}. 

Essentially, the graphs of order $k$ are generated from the graphs of order $k-1$ by recursively adding a new vertex and connecting it in all possible ways to the vertices already generated, pruning isomorphic copies on the way. 

To this basic algorithm we add two filtering/pruning steps:
\begin{enumerate}
\item Only those graphs of order $n-1$ are retained that are non-singular, cf.\ property 1 of Theorem~\ref{theo:nut}.
\item Only those graphs of order $n$ are retained that satisfy properties 2 and 3 of Theorem~\ref{theo:nut}.
\end{enumerate}
In addition, the inverse $A^{-1}$ of the adjacency matrix $A$ of a graph retained in step 1 above, is stored in memory so it can be reused for the second step. In general, a single graph of order $n-1$ gives rise to a large number (often hundreds or thousands) of graphs of order~Ò$n$ so this turns out to be a very effective optimisation.

To check whether the adjacency matrix $A$ is non-singular and then compute its inverse, the standard Gauss-Jordan algorithm from linear algebra can be used, which essentially amounts to computing the echelon form of the augmented matrix $(A\mid 1)$ to yield $(1\mid A^{-1})$. 

This algorithm has an asymptotic complexity of $O(n^3)$ under the assumption that all standard arithmetic operators take constant time. (In practice, however, this assumption is valid only with computer operations that use a fixed number of computer bits, in particular, not with multiprecision arithmetic).  
Since, as we mentioned above, it is essential to avoid rounding errors, the algorithm cannot make use of (finite precision) floating point operations. 

As the matrix $A$ has integral entries (in fact, it consists of only ones and zeros) we could instead use multiprecision rational arithmetic, i.e.\  work with exact fractions. Unfortunately, the numerators and denominators involved become very large very quickly, making this feasible only for graphs of small order. It is possible to adapt the algorithm so that
division can be avoided, i.e.\  work with multiprecision integers instead of rationals, but this does not improve the execution time significantly.

Instead we use a different approach, based on modular arithmetic, which turns out to be more efficient.

\subsection{Computations modulo $p$}
\label{subsection:modulo_p}
The main idea  is to perform most of the work using arithmetic  `modulo $p$' for some suitable primes $p$. Theorems \ref{theo-lift}--\ref{theo-lift-3} below show that we can then `lift' our results to the real numbers $\R$, provided we do this for a sufficient number of  primes $p$. We can choose the primes to be quite large, as long as they do not surpass the size of a word for which a computer can perform fast division and multiplication (say, $p \approx 2^{31}$ for present-day computers).

Denote by $\Delta_n$ the maximum absolute value of the (real) determinant of a 0-1 matrix of size $n\times n$. We have
\begin{theorem}[Hadamard]
\[
\Delta_n \le 2\left(\frac{n+1}4\right)^{\displaystyle{\frac {n+1}2}}.
\]
\end{theorem}
Equality is possible, e.g. $\Delta_3=2$ is reached by the adjacency matrix $\left(\begin{smallmatrix}0&1&1\\1&0&1\\1&1&0\end{smallmatrix}\right)$ of the complete graph $K_3$.

For small values of $n$, exact values of $\Delta_n$ have been obtained \cite{Orrick2005}~:
\[
\begin{array}{c|*{20}c}
n & 2 & 3 & 4 & 5 & 6 & 7 & 8   \\
\hline
\Delta_n & 1 & 2 & 3 & 5 & 9 & 32 & 56  \\
\\
n & 9 & 10 & 11 & 12& 13 & 14& 15\\
\hline
\Delta_n  & 144 & 320 & 1458 & 3645 &9477 & 25515& 131072\\
\\
n  & 16 & 17 & 18 & 19 & 20 \\
\hline
\Delta_n& 327680 & 1114112 & 3411968 & 19531250 & 56640625 \\
\end{array}
\]

Let $A$ be a matrix with integral entries. Let $p$ be a prime. We write $\det_{p} A$ (resp.\ $\rank_{p} A$, $\adj_{p} A$)
for the determinant (resp.\ rank, adjugate) of $A$ over the finite field $\F_{p}$. Recall that $\rank_{p} A$ is also called the $p$-rank of $A$. We have the following properties:

\begin{theorem}
\label{theo-lift}
Let $A$ be a symmetric 0-1-matrix of size $n\times n$. Let $p_1, p_2, \ldots p_k$ 
be distinct primes such that $p_1p_2\cdots p_k > \Delta_n$. 
Then $A$ is non-singular over $\R$  if and only if $A$ is non-singular over $\F_{p_i}$ for at least one $i$, $1\le i \le k$. Or equivalently: $\det A = 0$ if and only if $\det_{p_i} A = 0$ 
for all $i$, $1\le i \le k$. 
\end{theorem}
\begin{proof}
Because all entries of $A$ are integers, we have $\det_{p_i} A = \det A\bmod {p_i}$. Note that $\det_{p_i} A = 0$ if and only if
$\det A$ is divisible by $p_i$. 

Hence, if $\det A = 0$ then also $\det_{p_i} A = 0$, for all $i$. Conversely,
if $\det_{p_i} A = 0$ for all $i$, then $\det A$ must be divisible by all $p_i$, and hence by their product. Since $|\det A| \le \Delta_n < \prod p_i$, this
is only possible when $\det A = 0$.
\end{proof}

\begin{theorem}
\label{theo-lift-2}
Let $A$ be a symmetric 0-1-matrix of size $n\times n$. Let $p_1, p_2, \ldots p_k$ 
be distinct primes such that $p_1p_2\cdots p_k > \Delta_n$. 
Then $\rank A =n-1$ if and only if 
$\rank_{p_i} A \le n-1$ for all $i$, $1\le i \le k$, with equality in at least one case.
\end{theorem}
\begin{proof} 
We obtain $\adj_{p_i} A$ by reducing every entry 
of $\adj A$  modulo $p$. (Note that $\adj A$ has integral entries.)
Each entry $\adj A$ is the determinant of a 0-1-matrix
of size $n-1\times n-1$
(viz.\ the corresponding co-factor). Hence if $\prod p_i > \Delta_{n} \ge \Delta_{n-1}$,
an entry which is zero in $\adj_{p_i} A$ for all $i$, must correspond exactly to a zero entry in $\adj A$.

If $A$ has rank $n-1$, then $\det A = 0$ and $\adj A\ne 0$. Hence $\det_{p_i} A = 0$ and the rank of $A$  over $\F_{p_i}$ is at most $n-1$, for every $i$. Also $\adj_{p_i} A\ne 0$ for at least one $i$ and hence $\rank_{p_i} A = n-1$ in that case. Conversely, if $\rank_{p_i} > n-1$ for at least one $i$, then $\det A\ne 0$. Also if $\rank_{p_i} < n-1$ for all $i$, then $\adj_{p_i} A = 0$ for all $i$ and therefore $\adj A = 0$, which implies $\rank A < n-1$.
\end{proof}

\begin{theorem}
\label{theo-lift-adj}
Let $\Gamma$ be a graph of order $n$ with adjacency matrix $A$ of rank $n-1$.  Let $p_1, p_2, \ldots p_k$ 
be distinct primes such that $p_1p_2\cdots p_k > \Delta_n$. Then $\Gamma$ is a nut graph if and only if for each element position $(g,h), 1 \le g,h \le n$, there is at least one $i, 1\le i\le k$ such that $(\adj_{p_i} A)_{g,h} \ne 0$.
\end{theorem}
\begin{proof}
The proof runs along the same lines as the proof of Theorems \ref{theo-lift} and 
\ref{theo-lift-2}. By Lemma \ref{lemma-adj}, $\Gamma$ is a nut graph if and only if
$\adj A$ contains no zero entries. An entry of  $\adj A$ is zero if and only if
the corresponding entry of $\adj_{p_i} A$ is zero, for all $i, 1\le i\le k$.
\end{proof}

Theorems \ref{theo-lift} and \ref{theo-lift-2} provide fast ways to check whether the adjacency matrix of a graph has nullity 0 or 1, with a complexity of $O(n^3\log \Delta_n) \approx O(n^4\log n)$. Theorem \ref{theo-lift-adj} is less useful in the general case, because it involves computing the adjugate matrix. For small values of $n$ (such that $\Delta_n < 2^{32}$) we may however use the following
\begin{corollary}
\label{cor:mod_p}
Let $\Gamma$ be a graph of order $n$ with adjacency matrix $A$ of rank $n-1$.  Let $p$ be prime, $p > \Delta_n$. Then $\Gamma$ is a nut graph if and only if it is a nut graph `modulo $p$'.
\end{corollary}

Note that the condition $p > \Delta_n$ cannot simply be waived. Indeed, in the course of our experiments we found seven (IPR fullerene) graphs of between 278 
and 300 vertices that are nut graphs `modulo $p$' (with $p=2^{32}-5 = 4\,294\,967\,291$) 
but that in reality turned out to be non-singular (with a determinant divisible by $p$).

For larger graphs we can use the following variant of Theorem~\ref{theo-lift-adj}:
\begin{theorem}
\label{theo-lift-3}
Let $\Gamma$ be a graph of order $n$ with adjacency matrix $A$ of rank $n-1$.  Let $p_1, p_2, \ldots p_k$ 
be distinct primes such that $p_1p_2\cdots p_k > \Delta_n$ and such that $\rank_{p_i} A = n-1$. Then $\Gamma$ is a nut graph if and only if for each coordinate position $g$, $1 \le g \le n$, there is at least one $i, 1\le i\le k$ such that a corresponding (non-trivial) kernel vector $v_i$ of $A$, computed modulo $p_i$, has $(v_i)_g\ne 0$.
\end{theorem}
\begin{proof}
For a matrix $A$ of $p_i$-rank $n-1$ and non-trivial kernel vector $v_i$ satisfies $\adj_{p_i} A = \lambda_i v_i^T v_i$ for some scalar $\lambda_i\in\F_{p_i}$, $\lambda_i\ne 0$. (Cf.\ proof of Lemma \ref{lemma-adj}.) The theorem now follows from Theorem \ref{theo-lift-adj}.
\end{proof}

Note that this theorem requires the primes $p_1, p_2, \ldots$ to be chosen such that $\rank_{p_i} A = n-1$. Fortunately, such primes can always be found (and in fact, most primes will satisfy this property). Indeed, as $\rank A = n-1$, only primes are forbidden for which $\adj_{p_i} A = 0$ (while $\adj A \ne 0$). These are the primes that divide all elements of $\adj A$ at the same time. There is necessarily only a finite number of these, and 
this number is typically zero.

The algorithm to check whether a given graph $\Gamma$ of order $n$ with adjacency matrix~$A$ hence runs as follows.
Perform the following for a (fixed) sequence of distinct primes $p_1,p_2,\ldots$ with $p_i\lessapprox 2^{31}$.
\begin{enumerate}
\item Compute $\rank_{p_i} A$ using the standard Gauss-Jordan algorithm over $\F_{p_i}$.
\item If $\rank_{p_i} A = n$, stop the algorithm. $\Gamma$ is not a nut.
\item If $\rank_{p_i} A < n-1$, discard $p_i$ and turn to the next prime in the list.
\item Otherwise, compute a non-trivial kernel vector $v_i$ for $A$ over $\F_{p_i}$. (Such a vector can easily be obtained from the reduced matrix resulting from step 1 above.)
\item Proceed to the next prime.
\end{enumerate}
These steps should be repeated until one of the following occurs:
\begin{enumerate}
\renewcommand{\theenumi}{\Alph{enumi}}
\item The product of the discarded primes exceeds $\Delta_n$. In this case $\rank A < n-1$ and $\Gamma$ is not a nut.
\item The product of the non-discarded primes exceeds $\Delta_n$.
\end{enumerate} 
(The product of the primes need not be computed exactly. We may use an upper estimate for $\log_2 \Delta_n$ and divide this by 31 to obtain the required number of tries).

In case A, all primes that have been tried will also have been discarded. In case B, we still need to investigate the kernel vectors $v_i$ that were produced on the way. The graph $\Gamma$ will then be a nut if and only if there is no coordinate position where every vector $v_i$ is zero.

\section{Testing and results}

\subsection{The numbers of nut graphs}

\label{subsect:nut_counts}

We implemented our generation algorithm for nut graphs described in Section~\ref{section:generation_nuts} in the programming language C and incorporated it in the program \textit{geng}~\cite{nauty-website, mckay_14} which takes care of the isomorphism rejection. Our implementation of this algorithm is called \textit{Nutgen}, and can be downloaded from~\cite{nutgen-site}. 

In~\cite{fowler2014omni} all nut graphs up to 10 vertices were determined. Using \textit{Nutgen} we generated all non-isomorphic nut graphs up to 13 vertices and also went several steps further for nut graphs with a given lower bound on the girth. (The \textit{girth} is the length of the smallest cycle of a graph).
Table~\ref{table:number_of_nuts} shows the counts of the complete lists of nut graphs generated by our program. 
Figure~\ref{fig:smallest_nuts_girth} shows drawings of the smallest nut graphs with respect to their girth.

\begin{table}
\centering
\footnotesize
\setlength{\tabcolsep}{4pt}
	\begin{tabular}{|c || c | c | c | c | c | c | c | c |}
		\hline
		Order & Nut graphs & $g=3$ & $g = 4$ & $g = 5$ & $g = 6$ & $g = 7$ & $g = 8$ & $g \geq 9$\\
		\hline
$0-6$  & 0 &  0  &  0  &  0  &  0  &  0 &  0  &  0 \\
7  &  3  & 3 &  0  &  0  &  0  &  0 &  0  &  0\\
8  &  13 & 13  &  0  &  0  &  0  &  0 &  0  &  0\\
9  &  560 & 560  &  0  &  0  &  0  &  0 &  0  &  0\\
10  &  12 551 & 12 551  &  0  &  0  &  0  &  0 &  0  &  0\\
11  &  2 060 490 & 2 060 474  &  14  &  2  &  0  &  0 &  0  &  0\\
12  &  208 147 869 & 208 147 847 &  20  &  2  &  0  &  0 &  0  &  0\\
13  &  96 477 266 994  & 96 477 263 085 &  3 889  &  20  &  0  &  0 &  0  &  0\\
14  &  ?  & ? &  18 994   &  35  &  0  &  0 &  0  &  0\\
15  &  ? & ?  &  3 640 637  &  1 021  &  5  &  1 &  0  &  0\\
16  &  ? & ?  &  48 037 856  &  2 410  &  5  &  0 &  0  &  0\\
17  &  ? & ?  &  10 722 380 269  &  88 818  &  154  &  1 &  0  &  0\\
18  &  ? & ?  &  ?  &  341 360  &  139  &  0 &  0  &  0\\
19  &  ? & ?  &  ?  &  14 155 634  &  6 109  &  36 &  0  &  1\\
20  &  ? & ?  &  ?  &  82 013 360  &  6 660  &  8 &  0  &  0\\
		\hline
	\end{tabular}
\caption{The numbers of nut graphs. Columns with a header of the form $g = k$ 
list the numbers of nut graphs with girth~$k$ at each order.}
\label{table:number_of_nuts}
\end{table}

\begin{figure}[h!t]
    \centering
	\includegraphics[width=0.9\textwidth]{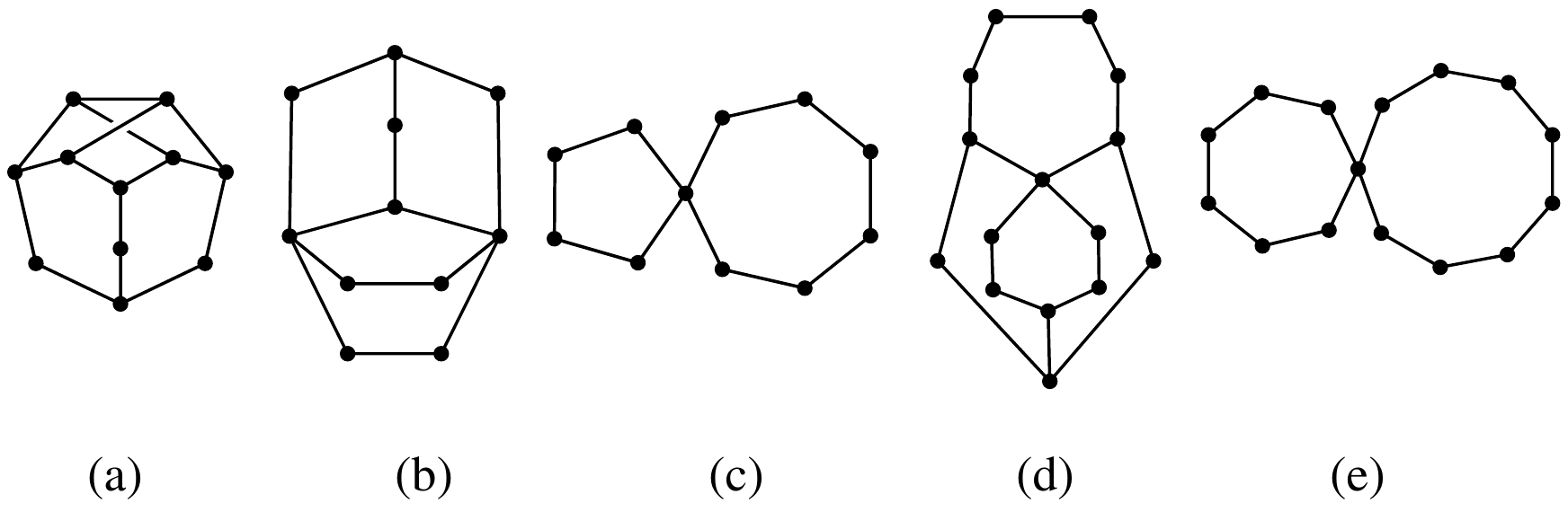}
    \caption{Small nut graphs: (a) one of the 14 smallest nut graphs with girth 4 (with 11 vertices); (b) and (c) the two smallest nut graphs with girth 5 (with 11 vertices);
(d) shows one of the 6 smallest nut graphs with girth 6 (with 15 vertices); (e) the smallest nut graph with girth 7 (with 15 vertices).}
    \label{fig:smallest_nuts_girth}
\end{figure}

In~\cite{fowler2014omni} all chemical nut graphs up to 16 vertices were determined.
Table~\ref{table:number_of_chemical_nuts} shows the counts of the complete lists of chemical nut graphs generated by our program and Figure~\ref{fig:smallest_chemical_nuts_girth} shows drawings of the smallest chemical nut graphs with respect to their girth.

The chemical relevance of girth is that rings of carbon atoms in networks are constrained by steric factors. 
The ideal bond angle for unsaturated ($sp^2$ hybridised) carbon atoms is $120 \degree$. Departures from a ring size of six are 
typically
punished with energy penalties that are especially severe for rings of size 3 and 4.
There are standard methods for comparing steric strain at an atom (vertex of the molecular graph)
(e.g.~\cite{haddon2001comment}), 
which can be related to mathematical notions of combinatorial curvature for polyhedra~\cite{fowler2015distributed}.

\begin{table}
\centering
\small

	\begin{tabular}{|c || c | c | c | c | c | c | c | c |}
		\hline
		Order & Nut graphs & $g=3$ & $g = 4$ & $g = 5$ & $g = 6$ & $g = 7$ & $g = 8$ & $g \geq 9$\\
		\hline
$0-8$  &  0  &  0  &  0  &  0  &  0 & 0  & 0 & 0 \\
9  &  1  &  1  &  0  &  0  &  0  & 0  & 0 & 0\\
10  &  0  &  0  &  0  &  0  &  0  & 0  & 0 & 0\\
11  &  8  &  7  &  1  &  0  &  0  & 0  & 0 & 0 \\
12  &  9  &  7  &  1  & 1  &  0  & 0  & 0 & 0\\
13  &  27  &  23  &  2  &  2  &  0  & 0  & 0 & 0\\
14  &  23  &  22  &  0  &  1  &  0  & 0  & 0 & 0\\
15  &  414  &  338  &  51  & 25  &  0  & 0  & 0 & 0\\
16  &  389  &  339  &  36  &  13  &  1  & 0  & 0 & 0\\
17  &  7 941  &  6 153  &  1 364  &  408  &  15  & 1  & 0 & 0\\
18  &  8 009  &  6 742  &  1 079  &  182  &  6  & 0  & 0 & 0\\
19 & 67 970 & 52 719 & 9 668 & 5 275 & 298 & 10 & 0 & 0\\
20 & 51 837 & 45 261 & 3 812 & 2 628 & 135 & 1 & 0 & 0\\
21 & 1 326 529 & 995 228 & 214 777 & 109 999 & 6 435 & 84 & 5 & 1\\
22 & 1 372 438 & 1 141 082 & 157 415 & 70 977 & 2 937 & 27 & 0 & 0\\
		\hline
	\end{tabular}
\caption{The numbers of chemical nut graphs. 
Columns with a header of the form $g = k$ 
list the numbers of chemical nut graphs with girth~$k$ at each order.}
\label{table:number_of_chemical_nuts}
\end{table}
\begin{figure}[h!t]
    \centering
	\includegraphics[width=0.9\textwidth]{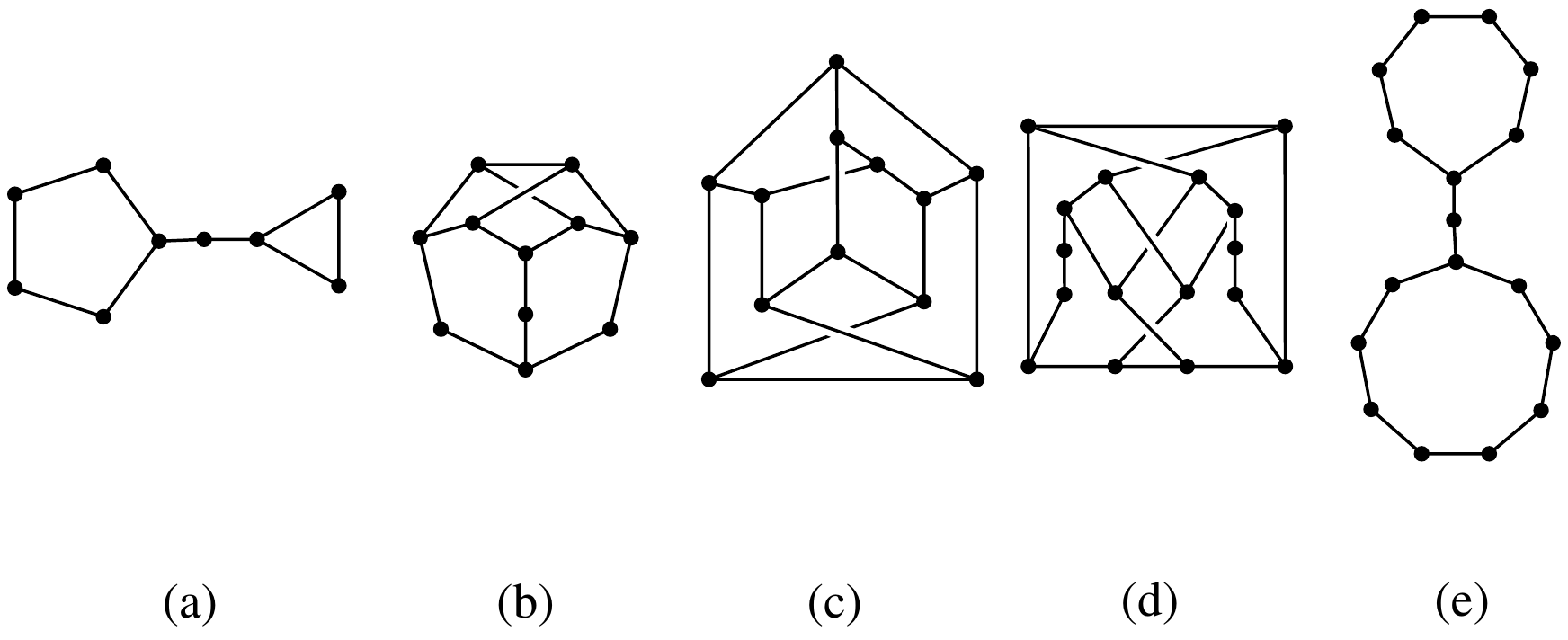}
    \caption{Figure~(a)-(e) show the smallest chemical nut graphs of girth 3, 4, 5, 6 and 7, respectively. They have 9, 11, 12, 16 and 17 vertices, respectively.}
    \label{fig:smallest_chemical_nuts_girth}
\end{figure}

In~\cite{sciriha2007nonbonding} Sciriha and Fowler determined all cubic polyhedral nuts (i.e.\ cubic planar 3-connected graphs that are also nut
graphs) up to 24 vertices. Using the program \textit{plantri}~\cite{brinkmann_07} and our program to test if a graph is a nut graph as a filter, we determined all cubic polyhedral nuts up to 34 vertices. The counts of these graphs can be found in Table~\ref{table:number_of_cubic_polyhedra_nuts}.
Perhaps the most interesting feature of that table is the emergence at $n = 26$ of examples where the number of vertices is not divisible by $6$ (see discussion in~\cite{sciriha2007nonbonding}).

\begin{table}
\centering
\small
	\begin{tabular}{| c | r | r |}
		\hline
		Order & cubic polyhedra & nut graphs\\
		\hline
4  &  1  &  0  \\
6  &  1  &  0  \\
8  &  2  &  0  \\
10  &  5  &  0  \\
12  &  14  &  2  \\
14  &  50  &  0  \\
16  &  233  &  0  \\
18  &  1 249  &  285  \\
20  &  7 595  &  0  \\
22  &  49 566  &  0  \\
24  &  339 722  &  62 043  \\
26  &  2 406 841  &  4  \\
28  &  17 490 241  &  316  \\
30  &  129 664 753  &  16 892 864  \\
32  &  977 526 957  & 3 676  \\
34  &  7 475 907 149  &  447 790  \\
		\hline
	\end{tabular}
\caption{The numbers of cubic polyhedral nut graphs.}
\label{table:number_of_cubic_polyhedra_nuts}
\end{table}

In~\cite{sciriha2007nonbonding} Sciriha and Fowler also determined all nut fullerenes up to 120 vertices and showed that there are no {\it IPR} nut fullerenes up to at least 150 vertices.

Using the program \textit{buckygen}~\cite{BGM12, GM15}, we determined all nut fullerenes up to 250 vertices, and showed that there are no nut IPR fullerenes up to at least 320 vertices. The numbers of nut fullerenes up to 250 vertices can be found in Table~\ref{table:number_of_fullerene_nuts}. Only 173 out of the 21 627 759 707 fullerene isomers up to 250 vertices are nuts.

\begin{table}
\centering
\small

	\begin{tabular}{|c | c |}
		\hline
		Order & Nut fullerenes\\
		\hline
36  &  1 \\
42  &  1 \\
44  &  1 \\
48  &  2 \\
52  &  2 \\
60  &  6 \\
72  &  2 \\
82  &  1 \\
\hline
\end{tabular}\quad
	\begin{tabular}{|c | c |}
		\hline
		Order & Nut fullerenes\\
		\hline
84  &  8 \\
96  &  5 \\
108  &  7 \\
120  &  5 \\
132  &  14 \\
144  &  6 \\
156  &  11 \\
160  &  1 \\
\hline
\end{tabular}\quad
	\begin{tabular}{|c | c |}
		\hline
		Order & Nut fullerenes\\
		\hline
168  &  11 \\
180  &  16 \\
192  &  8 \\
204  &  19 \\
216  &  9 \\
228  &  21 \\
240  &  16 \\
& \\
		\hline
	\end{tabular}

\caption{The numbers of nut fullerenes up to 250 vertices. 
For orders up to 250 where no count is listed,
the implication is that there is no nut fullerene of that order.}
\label{table:number_of_fullerene_nuts}

\end{table}

The nut graphs from Tables~\ref{table:number_of_nuts}-\ref{table:number_of_fullerene_nuts} can be
downloaded from the \textit{House of Graphs}~\cite{hog} at
\url{https://hog.grinvin.org/Nuts}

As a partial check on the correctness of our implementations and results, we compared our lists of nut graphs to the known lists of nut graphs up to 10 vertices and of chemical nut graphs up to 16 vertices, which were determined in~\cite{fowler2014omni}. Furthermore, we also compared our results on cubic polyhedral nuts and fullerene nuts with the results from~\cite{sciriha2007nonbonding}.
In each case all results were in complete agreement.

\subsection{Chemical properties of nut graphs}

\label{subsect:nut_chem_props}

In this section we describe the result of our computations of chemically relevant properties on the complete lists of nut graphs determined in Section~\ref{subsect:nut_counts}.

\subsubsection{Nut graphs and non-bonding orbitals}

The chemical significance of nuts comes from the association of the zero eigenvalue of the adjacency matrix with
\textit{non-bonding} orbitals (NBO) in molecules and with the \textit{Fermi level} in materials. 
A non-bonding orbital is balanced between stabilisation (bonding) and destabilisation (anti-bonding) of a molecule by the presence of an electron.
The Fermi level in a conductor corresponds to the energy
that separates occupied from empty bands of energy levels
at absolute zero.
The association between adjacency eigenvalues and eigenvectors and chemical concepts is straightforward: 
eigenvalues correspond to orbital energies, and eigenvectors to molecular orbitals. Of particular importance are 
the HOMO and LUMO ({\it highest occupied} and {\it lowest occupied} molecular orbitals, respectively).
In H{\"u}ckel theory of an all-carbon material, the Fermi level corresponds to the mean of HOMO and LUMO energies.

A neutral carbon network of $n$ centres has $n$ electrons distributed over the centres in delocalised
orbitals according to three rules: the Aufbau and Pauli principles and Hund's rule of maximum multiplicity. In short, electrons are assigned to eigenvectors in decreasing order of eigenvalue (Aufbau), with at most two electrons per eigenvector (Pauli) and, whenever a degeneracy (multiplicity) is encountered, electrons are spread out across eigenvectors/orbitals as far as possible, and with parallel spins as far as possible
(Hund's rule).
Each electron carries an up or down spin, with component along a fixed axis
$\pm {\frac{1}{2}}\hbar$,
where $\hbar$ is Planck's constant divided by $2\pi$; spin-up and spin-down possibilities are known as $\alpha$ and $\beta$, respectively. 
The \textit{occupation number} of a given orbital/eigenvector is therefore 2, 1 or 0, accordingly as it contains spin-paired electrons, a single electron or no electrons. 

The physical significance of occupation of an orbital is that the \textit{charge density} at each site is found in H\"uckel theory by summing the squares of eigenvector entries over each eigenspace, weighting the sum by the average occupation number of the space. Electrons with $\alpha$ and $\beta$ spin contribute equally to charge density. 
Likewise, the \textit{spin density} is determined by a
calculation involving squared eigenvector entries, but with $\alpha$ and $\beta$ spins contributing with opposite sign.
Thus, spin density is calculated from the squared entries in the eigenvectors, summed over any partially occupied eigenspace, weighted by the fraction 
$n_\alpha - n_\beta / (n_\alpha+n_\beta)$, where $n_{\alpha/\beta}$ is the number of electrons 
of $\alpha/\beta$ spin in the eigenspace.

Spin density has a particular chemical significance in that it indicates distribution of radical character. A radical is a molecule with unpaired electron spin distributed over its molecular framework. This has implications for reactivity and for physical measurements such as \textit{esr} (electron spin resonance) coupling constants~\cite{symons1978chemical}.
Radicals are typically reactive, and in the simplest picture, the most reactive sites within the radical will be those of highest spin density.

A nut graph has non-zero entries in the non-trivial nullspace vector on all vertices. It therefore corresponds in single occupation of the corresponding orbital to a distribution of spin density across the whole framework of unsaturated carbon atoms. Most radicals have a mixture of zero and non-zero spin densities across the framework. Nut graphs in this sense are the models for extreme delocalisation of spin density. As we will see below, chemical nut graphs in an electron configuration where the NBO is the singly occupied HOMO have at best a ratio of only 4 between highest and lowest spin densities. In a chemical non-nut graph these NBO spin densities are zero for some vertex or vertices.

The two parameters of chemical significance when considering a nut graph are 
therefore
the \textit{spectral position} and the \textit{dispersion} of entries of
the non-trivial nullspace vector of the nut graph.
The position of the zero eigenvalue in the spectrum of the adjacency matrix is important, because occupation of the NBO will have most effect on the properties of the $\pi$-system if the NBO is either the HOMO or the LUMO, indicating interesting spin-distributions in systems with total charges near to neutrality. A useful indicator for a nut graph is therefore $\delta q$, defined as the charge required for half occupation of the zero eigenvalue vector. If the nut graph has $n_+$ strictly positive and $n_-$ strictly negative eigenvalues, simple counting gives
$$\delta q = n_+ - n_- = n - 2n_+ -1 = 2n_- - n +1 $$
for the charge in units of $e$, the proton charge.
For example, the $9$-vertex chemical nut graph with spectrum 
$\theta_+,2,1,\phi\sp{-1},0,-1,\theta_-,-\phi,-2$,
where $\theta_\pm = {1\pm\sqrt{13}}/2$, carries $9$ electrons
and hence is neutral at half-filling of the NBO.

The defining characteristic of nut graphs is that all vertices carry a non-zero entry in the unique nullspace eigenvector, and hence spin density is distributed across the whole framework. All chemical nut graphs will have non-zero spin density at all sites, if the NBO is half occupied. A simple indicator of the dispersion of the spin-density distribution is the ratio $r$ of magnitudes of largest and smallest entries in the nullspace eigenvector. The function $r^2$ gives the expected ratio of spin densities
at most- and least-spin-rich sites in the molecule. The following is straightforward to prove.

\begin{theorem}\label{thm:r_at_least_2}
A chemical nut graph has $r^2 \geq 4$.
\end{theorem}

\begin{proof}
A chemical graph has maximum degree at most $3$, but a nut graph has minimum degree at least $2$ and is not a cycle, so every chemical nut graph has a
vertex, say $v$, of degree $3$. Call the kernel-vector entries on the neighbours of $v$ 
$\{a, b, c\}$; choose a normalisation such that $a = 1$, $b=x$, $c = -1-x$ with $x > 0$. 
Then, either $x \geq 1$     and $| c/a | = |1+x| \geq 2$, 
      or     $0 < x < 1$ and $| c/b | = |(1+x)/x| >2$. 
Hence, $r\sp2 \geq 4$ for any chemical nut graph.
\end{proof}

A further aspect of nut graphs for a specific subset of chemical graphs is the classification of regular cubic nut
graphs. This is based on the distribution of entries in the unique nullspace eigenvector on the three neighbours of each vertex.  A uniform nut graph has entries $a\{2, -1, -1\}$ around every vertex (where $a$ is a single scaling factor). A balanced graph has entries in ratio $\{2, -1, -1\}$, but with different scaling factors. All other cubic nut graphs are `just nuts'.
Dispersion is at a minimum for uniform nuts
($r^2 = 4$);
for a balanced nut  $r^2$ is a power of $2$;
for a simple nut (a `just nut') it can be large.
Increase of $r$ corresponds to a reduction of the delocalisation of spin and charge densities for the orbital.
  
Nut graphs also figure in a different application of graph theory in chemistry: the source-and-sink-potential (SSP) model~\cite{goyer2007source,pickup2008analytical} of ballistic molecular condition. 
It has recently emerged that in this model
the transmission of an electron through a $\pi$ framework at the Fermi level is determined by nullities of four graphs~\cite{fowler2009selection, fowler2009conduction}: $G, G-\bar L, G-\bar R$ and $G-\bar L-\bar R$. Here $G$ is the molecular graph and $\bar L$ and $\bar R$ are the vertices of $G$ that are next to the leads in the molecular circuit. 

In this model, nut graphs are \textit{strong omniconductors}, that is to say they conduct at the Fermi level irrespective of the choice of vertices $\bar L \neq \bar R$ or $\bar L = \bar R$~\cite{fowler2013omni}.  These have a special significance in that connection of a strong omniconductor 
molecule to two leads in any manner whatsoever leads to a conducting device, in the simple
(empty-molecule) SSP model.
It has been proved in~\cite{fowler2014omni} that nut graphs are exactly the strong omniconductors of nullity one.

\subsubsection{Position in spectrum of the zero eigenvalue}

Tables~\ref{table:NBO_general_nuts}-\ref{table:NBO_cubic_polyhedra} show frequency tables of the position of the zero eigenvalue within the spectrum of the general nut graphs, chemical nut graphs and cubic polyhedral nut graphs, respectively.

To determine the data on the NBO in Tables~\ref{table:NBO_general_nuts}-\ref{table:NBO_cubic_polyhedra}, the eigenvalues $\lambda_1,...,\lambda_n$ 
were sorted in descending order (i.e.\ $\lambda_1$ is the largest eigenvalue and $\lambda_n$ the largest).
The tables report numbers of cases where the zero eigenvalue is at position
${\left\lceil\frac n2\right\rceil} +k$ for the nut graphs of order $n$. For the sets of graphs and ranges of $n$ considered here, $k$ falls between
$-3$ and $+3$ ($k = +3$ for some fullerenes).

The chemical implication is that there are typically
many molecules based on
chemical nut graphs where a radical with fully
delocalised spin density and small net charge
would be produced by half-occupation of the kernel eigenvector:
for odd $n$, these have the NBO at position $\left\lceil\frac n2\right\rceil$ 
and correspond to the neutral molecule;
for even $n$, these have the NBO at position 
$\left\lceil\frac n2\right\rceil$ and charge $+1$ or at    
$\left\lceil\frac n2\right\rceil+1$ and charge $-1$.

However, the tests also showed that for all nut fullerenes up to 250 vertices 
the NBO is at position $k=3$, except for one fullerene on 42 vertices and 
one of 60 vertices where the NBO is at position $k=2$, so the charges at which 
nut fullerene radicals 
could display this delocalisation are further from neutrality.
Amongst chemical graphs, fullerenes are atypical in that they tend to have $n_+ > n_-$, and hence occupation of the zero eigenvalue
 often corresponds to a significant negative molecular charge.

\newcommand*{\nhalf}{\displaystyle{\left\lceil\frac n2\right\rceil}}
\begin{table}
\centering
\small
\footnotesize
	\begin{tabular}{|c || c | c | c | c | c | c || c |}
		\hline
Order \rule[-2ex]{0pt}{5.5ex}& $\nhalf \!-\! 3$ & $\nhalf\!-\!2$ &$\nhalf\!-\!1$ & $\nhalf$ & $\nhalf\!+\! 1 $ & $\nhalf\!+\!2$ & Total\\
		\hline
7 &  &  &  & 3 &  &  & 3 \\
8 &  &  &  & 13 &  &  & 13 \\
9 &  & 1 & 65 & 494 &  &  & 560 \\
10 &  & 4 & 295 & 12 169 & 83 &  & 12 551 \\
11 & 14 & 2 597 & 316 473 & 1 741 400 & 6 &  & 2 060 490 \\
12 & 55 & 29 313 & 8 879 721 & 196 259 526 & 2 979 253 & 1 & 208 147 869 \\
		\hline
	\end{tabular}

\caption{Frequency table of the position of the zero eigenvalue
within the spectrum of the nut graphs. 
The columns with a header of the form $\lceil n/2 \rceil + k$ contain the numbers of nut graphs of order $n$ where the NBO is at position $\lceil n/2 \rceil + k$.}
\label{table:NBO_general_nuts}

\end{table}

\begin{table}
\centering
\small

	\begin{tabular}{|c || c | c | c | c | c || c |}
		\hline
Order\rule[-2ex]{0pt}{5.5ex} & $\nhalf-2$ &$\nhalf- 1$ & $\nhalf$ & $\nhalf + 1 $ & $\nhalf + 2$ & Total\\
		\hline
9 &  &  & 1 &  &  & 1 \\
10 &  &  &  &  &  & 0 \\
11 &  &  & 8 &  &  & 8 \\
12 &  &  & 6 & 3 &  & 9 \\
13 &  &  & 27 &  &  & 27 \\
14 &  & 1 & 21 & 1 &  & 23 \\
15 &  & 5 & 409 &  &  & 414 \\
16 &  & 5 & 311 & 73 &  & 389 \\
17 & 3 & 173 & 7 754 & 11 &  & 7 941 \\
18 & 1 & 112 & 5 769 & 2 121 & 6 & 8 009 \\
19 & 22 & 1 140 & 66 766 & 42 &  & 67 970 \\
20 & 9 & 761 & 42 203 & 8 859 & 5 & 51 837 \\
21 & 194 & 18 986 & 1 306 168 & 1 181 &  & 1 326 529 \\
22 & 107 & 13 788 & 1 024 175 & 334 132 & 236 & 1 372 438 \\
		\hline
	\end{tabular}

\caption{Frequency table of the position of the zero eigenvalue 
within the spectrum of the chemical nut graphs. 
Column headings as in Table {\ref{table:NBO_general_nuts}}.}
\label{table:NBO_chemical_nuts}

\end{table}

\begin{table}
\centering
\small

	\begin{tabular}{|c || c | c | c | c | c || c |}
		\hline
Order\rule[-2ex]{0pt}{5.5ex} & $\nhalf-2$ &$\nhalf- 1$ & $\nhalf$ & $\nhalf + 1 $ & $\nhalf + 2$ & Total\\
		\hline
12 &  &  & 2 &  &  & 2 \\
18 &  & 7 & 262 & 16 &  & 285 \\
24 & 4 & 3 022 & 54 699 & 4 317 & 1 & 62 043 \\
26 &  &  & 1 & 2 & 1 & 4 \\
28 &  & 128 & 187 & 1 &  & 316 \\
30 & 18 486 & 1 363 546 & 14 169 947 & 1 339 896 & 989 & 16 892 864 \\
32 & 78 & 442 & 860 & 2 150 & 146 & 3 676 \\
34 & 108 & 197 257 & 249 825 & 600 &  & 447 790 \\
		\hline
	\end{tabular}

\caption{Frequency table of the position of the 
zero eigenvalue within the spectrum of the cubic polyhedral nut graphs. 
Column headings as in Table {\ref{table:NBO_general_nuts}}.}
\label{table:NBO_cubic_polyhedra}

\end{table}

\subsubsection{Ratio of the largest to smallest kernel eigenvector entry}

In this section we will tabulate and discuss the ratio of the largest to smallest entry in the eigenvector that 
corresponds with the zero eigenvalue for nut graphs (note: here we use the absolute values of the entries).

Since there are too many values to list the counts for each ratio and order, we only list the smallest ratio, largest ratio and the number of graphs with the smallest and largest ratio for each order. These results can be found in Tables~\ref{table:ratio_general}-\ref{table:ratio_cubic_polyhedra} for nut graphs, chemical nut graphs and cubic polyhedral nut graphs, respectively. We also found several nut graphs, chemical nut graphs and cubic polyhedral nut graphs for which the ratio is not an integer.

\begin{table}
\centering
\small
	\begin{tabular}{|c || c | c | c | c |}
		\hline
		\multirow{2}{*}{Order} & \multirow{2}{*}{min $r$} & frequency & \multirow{2}{*}{max $r$} & frequency\\	
		 &  & min $r$ &  & max $r$\\	
		\hline
7 & 1 & 3 & 1 & 3  \\
8 & 1 & 7 & 2 & 6  \\
9 & 1 & 83 & 4 & 4  \\
10 & 1 & 988 & 6 & 1  \\
11 & 1 & 34 910 & 12 & 9  \\
12 & 1 & 1 739 859 & 16 & 13  \\
		\hline
	\end{tabular}

\caption{Counts of the smallest ratio, largest ratio and the number of graphs with the smallest and largest ratio for nut graphs. (Here $r$ stands for the magnitude
of the ratio of largest to smallest entry in the eigenvector that corresponds with the zero eigenvalue).}
\label{table:ratio_general}
\end{table}

\begin{table}
\centering
\small
	\begin{tabular}{|c || c | c | c | c |}
		\hline
		\multirow{2}{*}{Order} & \multirow{2}{*}{min $r$} & frequency & \multirow{2}{*}{max $r$} & frequency\\	
		 &  & min $r$ &  & max $r$\\	
		\hline
9 & 2 & 1 & 2 & 1  \\
10 & - & - & - & -  \\
11 & 2 & 6 & 4 & 1  \\
12 & 2 & 9 & 2 & 9  \\
13 & 2 & 7 & 4 & 8  \\
14 & 2 & 9 & 4 & 6  \\
15 & 2 & 80 & 8 & 2  \\
16 & 2 & 195 & 4 & 73  \\
17 & 2 & 1 284 & 10 & 12  \\
18 & 2 & 4 151 & 10 & 1  \\
19 & 2 & 1 822 & 15 & 5  \\
20 & 2 & 3 872 & 13 & 2  \\
21 & 2 & 32 278 & 22 & 7  \\
22 & 2 & 149 748 & 18 & 4  \\
		\hline
	\end{tabular}

\caption{Counts of the smallest ratio, largest ratio and the number of graphs with the smallest and largest ratio for chemical nut graphs. Conventions as in 
Table~{\ref{table:ratio_general}}.}
\label{table:ratio_chemical}
\end{table}

An observation from Table~\ref{table:ratio_chemical} is that for every order $n$ in range for which there is a 
chemical nut graph, the bound $r = 2$ is realised (recall from Theorem~\ref{thm:r_at_least_2} that chemical nut graphs have $r \ge 2$). We now show that this is always the case. In fact, the statistical evidence suggests that $r = 2$ is a common value.  

\begin{theorem}
There is a chemical nut graph with $r = 2$ for every order $n \ge 9$ ($n \ne 10$).
\end{theorem}

\begin{proof}
An edge of a nut graph carries entries $a-b$ in the kernel eigenvector, and can be expanded by insertion of 
a $P_4$ unit to give a kernel eigenvector of the $(n+4)$-vertex graph with entries
$a-b-{\overline a}-{\overline b}-a-b$; the expanded graph is still a nut, and as no new entry magnitudes have been
created, $r$ is conserved. 
Hence, to guarantee existence of chemical nut graphs with $r = 2$ for all $n \ge 9$ ($n \ne 10$)
it is sufficient to have one such graph at each of $9, 11, 12, 14$, which is guaranteed by the data in
Table~\ref{table:ratio_chemical}. 
\end{proof}

\begin{table}
\centering
\small
	\begin{tabular}{|c || c | c | c | c |}
		\hline
		\multirow{2}{*}{Order} & \multirow{2}{*}{min $r$} & frequency & \multirow{2}{*}{max $r$} & frequency\\	
		 &  & min $r$ &  & max $r$\\	
		\hline
12 & 2 & 2 & 2 & 2  \\
18 & 2 & 235 & 7 & 2  \\
24 & 2 & 35 632 & 20 & 2  \\
26 & 4 & 2 & 12 & 1  \\
28 & 5 & 7 & 18 & 1  \\
30 & 2 & 6 535 314 & 52 & 1  \\
32 & 4 & 803 & 25 & 1  \\
34 & 4 & 860 & 49 & 2  \\
		\hline
	\end{tabular}

\caption{Counts of the smallest ratio, largest ratio and the number of graphs with the smallest and largest ratio for cubic polyhedral nut graphs. Conventions as in 
Table~{\ref{table:ratio_general}}}.
\label{table:ratio_cubic_polyhedra}
\end{table}

The chemical nut graphs for which the NBO is at position $\lceil n/2 \rceil$ and which have the smallest ratio (i.e.\ 2) are of special chemical interest since these will have the smoothest distribution of spin density in a molecular graph with an electron count close to neutrality where this eigenvector
is half occupied. These counts are listed in Table~\ref{table:NBO_n2_ratio2}.
Conversely, nut graphs with maximum $r$ are the nut graphs that are in a sense 
as close as possible to losing their nut status.

\begin{table}
\centering
\small

\begin{tabular}{|c | c |}
		\hline
		Order & Counts\\
		\hline
9 & 1 \\
10 & - \\
11 & 6 \\
12 & 6 \\
13 & 7 \\
14 & 7 \\
15 & 77 \\
\hline
\end{tabular}\quad
	\begin{tabular}{|c | c |}
		\hline
		Order & Counts\\
		\hline
16 & 142 \\
17 & 1 188 \\
18 & 2 753 \\
19 & 1 656 \\
20 & 2 773 \\
21 & 29 932 \\
22 & 98 087 \\
\hline
\end{tabular}

\caption{The number of chemical nut graphs for which the NBO is at position $\nhalf$  
and the ratio $r$ of the largest to smallest entry in the eigenvector 
which corresponds with the zero eigenvalue is minimum (i.e.\ $r=2$).}
\label{table:NBO_n2_ratio2}

\end{table}

\medskip

\noindent
\textit{Acknowledgements:}
Patrick W.\ Fowler  is supported by the University of Sheffield and the Royal Society/Leverhulme Foundation. Jan Goedgebeur is supported by a Postdoctoral Fellowship of the Research Foundation Flanders (FWO).
Most computations for this work were carried out using the Stevin Supercomputer Infrastructure at Ghent University.


\end{document}